\documentclass[11pt]{amsart}
\usepackage{geometry}
\usepackage{graphicx}
\usepackage{amssymb}
\usepackage{amsmath, mathrsfs, stmaryrd, color}
\newtheorem{theorem}{Theorem}
\newtheorem{remark}[theorem]{Remark}
\newtheorem{definition}[theorem]{Definition}
\newtheorem{proposition}[theorem]{Proposition}
\newtheorem{corollary}[theorem]{Corollary}
\newtheorem{lemma}[theorem]{Lemma}
\newcommand{\R}{\mathbb R}
\numberwithin{theorem}{section}
\numberwithin{equation}{section}
\begin{document}

\title[AdS rest mass]{The rest mass of an asymptotically  Anti-de Sitter spacetime}
\author{Po-Ning Chen, Pei-Ken Hung, Mu-Tao Wang, and Shing-Tung Yau}
\date{December 27, 2015}
\begin{abstract}  We study the space of Killing fields on the four dimensional AdS spacetime $AdS^{3,1}$. Two subsets $\mathcal{S}$ and $\mathcal{O}$ are 
identified: $\mathcal{S}$ (the spinor Killing fields) is constructed from imaginary Killing spinors, and $\mathcal{O}$ (the observer Killing fields) consists of all hypersurface orthogonal, future
timelike unit Killing fields. When the cosmology constant vanishes, or in the Minkowski spacetime case, these two subsets have the same convex hull in the space of Killing fields. In presence
of the cosmology constant, the convex hull of  $ \mathcal{O}$ is properly contained in that of $\mathcal{S}$. This leads to two different notions of 
energy for an asymptotically AdS spacetime, the spinor energy and the observer energy. In \cite{Chrusciel-Maerten-Tod}, Chru\'sciel, Maerten and Tod proved the positivity of the spinor energy and derived important consequences among the related conserved quantities.  We show that the positivity of the
observer energy follows from the positivity of the spinor energy. A new notion called the ``rest mass" of an asymptotically AdS spacetime is then defined by minimizing the observer energy, and is shown to be evaluated in terms of the adjoint representation of the Lie algebra of Killing fields.  It is proved that the rest mass has the desirable rigidity property that characterizes the AdS spacetime. 
\end{abstract}
\thanks{P.-N. Chen is supported by NSF grant DMS-1308164, M.-T. Wang is supported by NSF grants DMS-1105483 and DMS-1405152,  and S.-T. Yau is supported by NSF
grants  PHY-0714648 and DMS-1308244. This work was partially supported by a grant from the Simons Foundation (\#305519 to Mu-Tao Wang). Part of this work was carried out
when P.-N. Chen and M.-T. Wang were visiting the Department of Mathematics and the Center of Mathematical Sciences and Applications at Harvard  University.} 
\maketitle
\section{Introduction}
In special relativity, the rest mass of a particle is defined to be the minimum value of energy seen among all observers, namely, the collection of all future timelike unit Killing vector fields in the Minkowski spacetime. While the notion of mass is a more challenging question in general relativity, the positivity of  the ADM mass \cite{ADM} was proved by Schoen-Yau \cite{Schoen-Yau1,Schoen-Yau2} and Witten \cite{Witten} for an asymptotically flat initial data set. 

\begin{theorem}\label{flat_pmt}
Let $(M, g, k)$ be an asymptotically flat initial data set that satisfies the dominant energy condition. The ADM mass of $(M, g, k)$ is non-negative. It vanishes if and only if $(M, g, k)$ is the data of a hypersurface in the Minkowski spacetime $\mathbb{R}^{3,1}$.
\end{theorem}

We do not emphasize the exact asymptotics at infinity. The optimal result is found by Bartnik for the time symmetric case \cite{Bartnik} and by  Chru\'sciel for the general case \cite{Chrusciel}.

Witten's proof employs the spinor method and in particular it is shown that the energy associated with each asymptotically constant spinor is non-negative. 
The ADM mass, as the minimum of the energy, is thus non-negative.  An underlying algebraic fact of the proof is that the set of all constant spinors, through a quadratic relation, reproduces all future timelike null  translating Killing fields in $\R^{3,1}$. Out of the ten dimensional Killing fields of $\R^{3,1}$, the four dimensional subspace of translating Killing fields corresponds exactly to the notion of ADM energy and linear momentum, and the complement of which corresponds to the Lorentz algebra $\underline{so}(3,1)$.  

The situation of an asymptotically AdS spacetime is more complicated, where the space of Killing fields is the ten dimensional Lie algebra $\underline{so}(3,2)$.  The concept of energy in this case was first defined by Abbott-Deser in  \cite{Abbott-Deser}, according to which, ``It also appears likely that the recent proof by Witten\footnote{The recent proof by Witten refers to  Witten's proof of the positive energy theorem \cite{Witten}. } that the energy is positive for gravity with $\Lambda=0$ can be extended to the present case." In \cite{Chrusciel-Maerten-Tod} (see also \cite{Abbott-Deser},  \cite{Ashtekar-Magnon},  \cite{Gibbons-Hull-Warner}, \cite{Henneaux-T}, \cite{Maerten}, \cite{Wang-Xie-Zhang}), Chru\'sciel, Maerten and Tod generalized Witten's proof and proved a positive energy theorem for an asymptotically AdS
initial data set. In particular, each asymptotic imaginary Killing spinor defines a non-negative energy. They also identified 10 physical quantities $(e, \vec{p}, \vec{c}, \vec{j})$ from the boundary terms of Witten's formula as energy, linear momentum, center of mass, and angular momentum. These conserved quantities determine a quadratic form $Q$ \cite[page 11]{Chrusciel-Maerten-Tod} on $\mathbb{C}^4$. The positive energy theorem states that $Q$ is positive semi-definite. An important consequence of their positive energy theorem is a set of inequalities among these conserved quantities.  They also proved rigidity theorems under various conditions involving the vanishing of the energy associated with asymptotic Killing spinors. Before \cite{Chrusciel-Maerten-Tod}, Gibbons-Hull-Warner in \cite{Gibbons-Hull-Warner} derived inequalities among $(e, \vec{p}, \vec{c}, \vec{j})$ under the stronger assumption that $Q$ is strictly positive, or there is no Killing spinor.

A natural question is how the energy defined by spinors can be interpreted as the energy seen by observers. In the asymptotically flat case, the ADM energy momentum $(\bar{e}, \bar{p}_j)$ can be considered as an element in the dual space of $\R^{3,1}$,  and the pairing with a future timelike unit  Killing field gives the energy seen by the corresponding observer. The ADM mass, $\sqrt{\bar{e}^2-\sum_j \bar{p}_j^2}$, is the minimum energy seen among all observers and characterizes the rigidity property. In attempt to generalize the quasi-local mass defined in \cite{Wang-Yau2, Wang-Yau3} from the Minkowski reference to the AdS reference, such an interpretation of energy \cite{Wang} appears to be crucial.  In this article, we show that the above picture holds true even in the asymptotically AdS case, if the ADM mass is replaced by the {\it rest mass} defined below.

For $\kappa\geq 0$, we recall the metric on the four dimensional AdS spacetime,  $AdS^{3,1}$:
\begin{align*}
-(1+\kappa r^2) {d}t^2+\frac{1}{(1+\kappa r^2)} {d}r^2+r^2({d}\theta^2+\sin^2\theta d\phi^2) .
\end{align*}

We denote the corresponding Lorentzian product by $\langle \cdot, \cdot\rangle_{AdS^{3,1}}$. 
Most previous discussions of asymptotically AdS energy \cite{CCS,Chrusciel-Maerten-Tod,Maerten,Xie-Zhang} are with respect to 
$\frac{\partial}{\partial t}$, a future timelike translating Killing field. However, from an invariant point of view, any future Killing field in the orbit of $\frac{\partial}{\partial t}$ under the $SO(3,2)$ action should be included
in the consideration of energy. These Killing fields can be characterized as follows.

\begin{definition}\label{observer_Killing}
A Killing field $K$ of $AdS^{3,1}$ is said to be an observer Killing field if $K$ is hypersurface orthogonal, future timelike, and ``unit" in the sense that 
the minimum of $-\langle K, K\rangle_{AdS^{3,1}}$ is $1$. 
The subset of all observer Killing fields of $AdS^{3,1}$ is denoted by $\mathcal{O}$. 
\end{definition}

It is clear that $-\langle \frac{\partial}{\partial t}, \frac{\partial}{\partial t}\rangle_{AdS^{3,1}}=\frac{1}{1+\kappa r^2}$ and $\frac{\partial}{\partial t}$ is ``unit" in the above sense. We shall see that the minimum of $-\langle K, K\rangle_{AdS^{3,1}}$
is exactly the norm square of $K$ measured by the Killing form on the Lie algebra of Killing fields, thus justifying the nomenclature ``unit".

On the other hand, an imaginary Killing spinor (see Definition \ref{Killing_spinor}) on $AdS^{3,1}$ gives rise to a Killing field: 

\begin{definition} \label{spinor_Killing} A Killing field $K$ of $AdS^{3,1}$ is said to be a spinor Killing field if there exists an imaginary 
Killing spinor $\psi$ such that 
\[\langle K, X\rangle_{AdS^{3,1}}=-(X.\psi,\psi)\] for any tangent vector $X$ of $AdS^{3,1}$, where $X.$ denote the Clifford multiplication by $X$.
The set of spinor Killing fields of $AdS^{3,1}$ is denoted by $\mathcal{S}$. 
\end{definition}

As the space of Killing fields is isomorphic to the Lie algebra $\underline{so}(3,2)$, both $\mathcal{K}$ and $\mathcal{O}$ are subsets $\underline{so}(3,2)$.
When $\kappa=0$ (the Minkowski spacetime), the boundary of $\mathcal{K}$ is $\mathcal{O}$, and they have the same convex hull. This is no longer true when $\kappa>0$ and we prove: 

\begin{theorem}\label{convex}
If $\kappa>0$, the convex hull of  $ \mathcal{O}$ in $\underline{so}(3,2)$ is properly contained in that of $\mathcal{S}$.
\end{theorem}
Our approach is to express the two subsets in terms of a special coordinate system on $\underline{so}(3,2)$ and study their defining equations. 
For simplicity, we assume $\kappa=1$ in the rest of the paper. 

Combining Theorem \ref{convex} and the positive energy theorem for asymptotically AdS spacetimes  proved by Chru\'sciel, Maerten and Tod in \cite{Chrusciel-Maerten-Tod}, we conclude that the energy seen by an observer (the observer energy) in  $ \mathcal{O}$ is non-negative. The rest mass $m$ of an asymptotically AdS spacetime is then defined by minimizing the observer energy (see Definition \ref{rest_mass}).  We show that the rest mass has the desirable rigidity property that characterizes the AdS spacetime. The following is the AdS version of Theorem \ref{flat_pmt}.

\begin{theorem}\label{pmt}
Let $(M, g, k)$ be an asymptotically AdS initial data set that satisfies the dominant energy condition as in \cite{Chrusciel-Maerten-Tod} . The rest mass of $(M, g, k)$ is non-negative. It vanishes if and only if $(M, g, k)$  is the data of a hypersurface in Anti-de Sitter spacetime $AdS^{3,1}$.
\end{theorem}
\begin{remark}
In $\R^{3,1}$, the set of observer Killing fields consists of unit future directed timelike vectors (or the interior of the future light cone in $\R^{3,1}$) and the set of spinor Killing fields consists of future directed null vectors (or the future light
cone in $\R^{3,1}$). It is clear that the convex hull of the former is contained in the convex hull of the latter. For an asymptotically flat initial data set, the minimum energy  associated with the observer Killing fields is the ADM mass $\sqrt{\bar{e}^2-\sum_j \bar{p}_j^2}$ while the minimum energy associated with the asymptotically constant spinors of unit length is $\bar e-|\bar p|$. Both quantities characterize the rigidity property of $\R^{3,1}$. For $AdS^{3,1}$, the sets of observer Killing fields and spinor Killing fields are described in Section 3 and 4, respectively. For an asymptotically AdS initial data set, the minimum energy associated with the observer Killing fields is the rest mass, which characterizes the AdS spacetime, see Theorem 1.5. The minimum energy associated with the asymptotically Killing spinors is computed in \cite{Wang-Xie-Zhang}, which does NOT  characterize the AdS spacetime.
\end{remark}
Theorem \ref{pmt} is proved at the end of Section 6. 
We remark that the positive mass theorem in  the time symmetric case ($k=0$), which corresponds to an asymptotically hyperbolic Riemannian manifold, has been considered by many authors \cite{Andersson-Cai-Galloway,Chrusciel-Herzlich,Chrusciel-Jezierski,XWang,Zhang}, who made use of sections of the complex rank 2 spinor bundle on the 3-dimensional hyperbolic space.  Chru\'sciel, Maerten and Tod in \cite{Chrusciel-Maerten-Tod} considered the full spacetime
version of the complex rank 4 spinor bundle on $AdS^{3,1}$.

\section{AdS spacetime} 
We review the AdS spacetime in different coordinate systems in this section. 
Take $\mathbb{R}^{3,2}$ with the coordinate system $(y^0, y^1, y^2, y^3, y^3)$ and the metric \[-(dy^4)^2+\sum_{i=1}^3(dy^i)^2-(dy^0)^2.\] $AdS^{3,1}$ can be identified with the timelike hypersurface
given by $-(y^4)^2+\sum_{i=1}^3(y^i)^2-(y^0)^2=-1$. Note that the group $SO(3,2)$ leaves this hypersurface invariant and thus the isometry group of $AdS^{3,1}$ is $SO(3,2)$, which is $10$ dimensional.

Consider the following parametrization of $AdS^{3,1}$:
\begin{align*}
y^0&=\sqrt{1+r^2}\sin t\\
y^1&=r\sin\theta\sin\phi\\
y^2&=r \sin\theta \cos\phi\\
y^3&=r \cos\theta\\
y^4&=\sqrt{1+r^2}\cos t.
\end{align*}
This gives the static chart of $AdS^{3,1}$
\begin{align*}
-(1+r^2) {d}t^2+\frac{1}{(1+r^2)} {d}r^2+r^2(d\theta^2+\sin^2\theta d\phi^2).
\end{align*}

In order to interpret the conserved quantities corresponding Killing fields, we consider a conformal coordinate system $x^0, x^1, x^2, x^3$ such that 
\[y^0=-\frac{x^0}{u}, y^i=-\frac{x^i}{u}, y^4=\frac{2}{u}+1\] with \[u=\frac{1}{4}\left[\sum_{k=1}^3(x^k)^2-(x^0)^2\right]-1.\]

We check that \[-(y^4)^2+\sum_{i=1}^3(y^i)^2-(y^0)^2=-(\frac{2}{u}+1)^2+\sum_{i=1}^3\frac{(x^i)^2}{u^2}-\frac{(x^0)^2}{u^2}=-1.\]

The induced metric on $AdS^{3,1}$ in the $x^0, x^1, x^2, x^3$ coordinate system is
\[\frac{1}{u^2}[-(dx^0)^2+\sum_{i} (dx^i)^2].\]

This is conformal to the Minkowski metric on $\mathbb{R}^{3,1}$. We deduce that $AdS^{3,1}$ has $10$ dimensional Killing fields and $15$ dimensional conformal Killing fields, exactly the same as $\mathbb{R}^{3,1}$.

In the $x^0, x^1, x^2, x^3$ coordinate system, $y^0\frac{\partial}{\partial y^i}+y^i\frac{\partial}{\partial y^0}$ corresponds to $x^0\frac{\partial}{\partial x^i}+x^i\frac{\partial}{\partial x^0}$, which is a 
boost Killing fields on $\mathbb{R}^{3,1}$.  
Indeed, we can check directly that 
\[y^k\frac{\partial}{\partial y^0}+y^0\frac{\partial}{\partial y^k}=x^k\frac{\partial}{\partial x^0}+x^0\frac{\partial}{\partial x^k}, k=1, 2, 3\] using $\frac{\partial u}{\partial x^0}=-\frac{1}{2} x^0$ and $\frac{\partial u}{\partial x^i}=\frac{1}{2} x^i$.

The rotation Killing fields $y^i\frac{\partial}{\partial y^j}-y^j\frac{\partial}{\partial y^i}, i<j$ correspond to $x^i\frac{\partial}{\partial x^j}-x^j\frac{\partial}{\partial x^i}$.

The time translating Killing field
\[y^0\frac{\partial}{\partial y^4}-y^4\frac{\partial}{\partial y^0}\]  corresponds to 
\[\frac{1}{4}\left([-(x^0)^2+\sum(x^k)^2]\frac{\partial}{\partial x^0}+2x^0(x^0\frac{\partial}{\partial x^0}+x^l\frac{\partial}{\partial x^l})\right)+\frac{\partial}{\partial x^0},\] the first summand corresponds to a conformal Killing field on $\mathbb{R}^{3,1}$. 

Finally, the second set of boost Killing fields
\[y^4\frac{\partial}{\partial y^i}+y^i\frac{\partial}{\partial y^4}, i=1, 2, 3,\] corresponds to 

\[\frac{1}{4}\left(-[-(x^0)^2+\sum(x^k)^2]\frac{\partial}{\partial x^i}+2x^i(x^0\frac{\partial}{\partial x^0}+x^l\frac{\partial}{\partial x^l})\right)-\frac{\partial}{\partial x^i}, i=1, 2, 3.\]

\section{Observer Killing fields of $AdS^{3,1}$}
Recall from the last section, a Killing  field of $AdS^{3,1}$ (or $\R^{3,2}$) can be written as 
\begin{equation}\label{Killing}
K=  A (y^0\frac{\partial}{\partial y^4}-y^4\frac{\partial}{\partial y^0} )+ B_i( y^0\frac{\partial}{\partial y^i}+y^i\frac{\partial}{\partial y^0})+ C_j (y^4\frac{\partial}{\partial y^j}+y^j\frac{\partial}{\partial y^4})+ D_p \epsilon_{pqr}y^q\frac{\partial}{\partial y^r}.  \end{equation}
For simplicity, we will write $K=(A,\vec{B},\vec{C},\vec{D})$ and consider $\vec{B}, \vec{C}$, and $\vec{D}$ as vectors in $\mathbb{R}^3$ with the usual 
inner and cross products. In particular, the time translating Killing field $\frac{\partial}{\partial t}$ is simply 
\[ \frac{\partial}{\partial t}= y^0\frac{\partial}{\partial y^4}-y^4\frac{\partial}{\partial y^0} \] and $(A, \vec{B}, \vec{C}, \vec{D})=(1, \vec{0}, \vec{0}, \vec{0})$.  Finally, let $\epsilon_{ijk}$ be the Levi-Civita symbol of $\R^3$.

In this section,we derive the condition on $A$, $\vec{B}$, $\vec{C}$ and $\vec{D}$ so that $K=(A,\vec{B},\vec{C},\vec{D})$  is an observer Killing field
defined in Definition \ref{observer_Killing}.

\begin{lemma}
The Killing field $K$ of the form \eqref{Killing}
is hypersurface orthogonal if and only if 
\[ \vec{B}\times \vec{C}=- A\vec{D}.\]
\end{lemma}
\begin{proof}
Let $\alpha$ be a one-form dual to the Killing field $K$. Thus
\begin{equation}\label{alpha}
\alpha=A(-y^0dy^4+y^4dy^0)+B_i(y^0 dy^i-y^idy^0)+C_j(y^4dy^j-y^jdy^4)+D_p \epsilon_{pqr} y^q dy^r.
\end{equation}

By the Frobenius theorem, $K$ is hypersurface  orthogonal if and only if  \begin{equation} \alpha\wedge d\alpha=0. \end{equation} 

We compute
\begin{equation}\label{d_alpha} d\alpha=-2A dy^0\wedge dy^4+2B_idy^0\wedge dy^i+2C_j dy^4\wedge dy^j+D_p\epsilon_{pqr} dy^q\wedge dy^r\end{equation} and rewrite
\[\alpha=(Ay^4-B_i y^i) dy^0+(-Ay^0-C_j y^j) dy^4+(B_i y^0+C_i y^4+D_k\epsilon_{kji} y^j) dy^i.\]

Expanding $\alpha\wedge d\alpha$ and collecting terms, we derive
\[\begin{split}\alpha\wedge d\alpha=& 2(-B_i C_j+B_j C_i-AD_k \epsilon_{kij} ) y^idy^0\wedge dy^4\wedge dy^j\\
&+[(Ay^4-B_l y^l) D_k\epsilon_{kij}-(C_i y^4+D_k \epsilon_{kli} y^l)(2B_j)] dy^0\wedge dy^i \wedge dy^j\\
&+[(-Ay^0-C_l y^l) D_k\epsilon_{kij}-(B_i y^0+D_k \epsilon_{kli} y^l)(2C_j)] dy^4\wedge dy^i \wedge dy^j\\
&+(B_i y^0+C_i y^4)(D_p\epsilon_{pqr}) dy^i \wedge dy^q\wedge dy^r.
\end{split}\]
As a result, $\alpha\wedge d\alpha=0$ is equivalent to the following set of conditions:
\begin{align}
\label{orthogonal_1}-B_i C_j+B_j C_i-AD_k \epsilon_{kij}  = & 0\\
\label{orthogonal_2} (Ay^4-B_l y^l) D_k\epsilon_{kij}-(C_i y^4+D_k \epsilon_{kli} y^l)B_j + (C_j y^4+D_k \epsilon_{klj} y^l)B_i =& 0\\
\label{orthogonal_3}   (Ay^0+C_l y^l) D_k\epsilon_{kij}+(B_i y^0+D_k \epsilon_{kli} y^l)C_j -(B_j y^0+D_k \epsilon_{klj} y^l)C_i   = & 0\\
\label{orthogonal_4} (B_i y^0+C_i y^4)(D_p\epsilon_{pqr}) \epsilon_{iqr} =& 0.
\end{align}
Equation \eqref{orthogonal_1} is the same as $\vec{B}\times \vec{C}=-A\vec{D}$. For equation \eqref{orthogonal_2}, the coefficient for $y^4$ also vanishes if  $\vec{B}\times \vec{C}=-A\vec{D}$. The coefficient for $y^l$ vanishes if 
\[ B_l D_k\epsilon_{kij} + D_kB_j \epsilon_{kli} -  D_kB_i \epsilon_{klj}=0. \]
It suffices to check this condition for $i=1$, $j=2$ and $l=1,2,3$. For this pair of $i$ and $j$, the equality holds trivially for $l=1$ and $l=2$. For $l=3$, the equation reduces to $\vec{B} \cdot \vec{D}=0$ which is also implied by  $\vec{B}\times \vec{C}=-A\vec{D}$. 

Similarly, equation \eqref{orthogonal_3} follows from  $\vec{B}\times \vec{C}=-A\vec{D}$.  Lastly,  equation \eqref{orthogonal_4} is the same as $\vec{B} \cdot \vec{D}=0$ and $\vec{C} \cdot \vec{D}=0$ and both follow from  $\vec{B}\times \vec{C}=-A\vec{D}$. 
\end{proof}
In the next lemma, we find the condition on $A$, $\vec{B}$, $\vec{C}$ and $\vec{D}$ so that $K$ is everywhere timelike.
\begin{lemma}\label{lemma_timelike}
Let $K$ be a hypersurface orthogonal Killing field of the form \eqref{Killing}
Then $K$ is timelike in $AdS^{3,1}$ if and only if 
\begin{equation}\label{positive_condition}
\begin{split}
|A| > & \max \{ |\vec{B}|, |\vec{C}|, |\vec{D}| \} \\
A^2 +  |\vec{D}|^2 > & |\vec{B}|^2+|\vec{C}|^2.
\end{split}
\end{equation}
Moreover, for such a hypersurface orthogonal and timelike Killing field, we have 
\[  \min  (- \langle K, K\rangle_{AdS^{3,1}}) =A^2 +  |\vec{D}|^2  -  |\vec{B}|^2- |\vec{C}|^2.  \]

\end{lemma}

It can be checked that the expression $A^2 +  |\vec{D}|^2  -  |\vec{B}|^2- |\vec{C}|^2$ is exactly the norm square of $K$ in the Killing form of 
$\underline{so}(3,2)$. 
\begin{remark}
It is easy to see that $|A| >  \max \{ |\vec{B}|, |\vec{C}| \}$ is a necessary condition for $K$ to be everywhere timelike.  On the other hand, we will see that $|A| >|\vec{D}|$ and 
\[ A^2 +  |\vec{D}|^2 - |\vec{B}|^2-|\vec{C}|^2>0 \]
are preserved under the action of isometry on Killing fields.
\end{remark}
\begin{proof}
We consider two cases depending on whether $\vec{D}$ vanishes. In each case, we apply coordinate  change to transform the Killing  field into a simpler form.

Suppose $\vec{D}$ is not zero. We show that $(A, \vec{B},\vec{C},\vec{D})$ can be changed into the following form:
\[ (A, \vec{B},\vec{C},\vec{D})= (A,(B,0,0),(0,C,0),(0,0,D))  \]
with $AD+BC= 0$.

First, it is easy to see that starting from 
\[ (A, \vec{B},\vec{C},\vec{D})= (A,(B,0,0),(0,C,0),(0,0,D))  \]
with $AD+BC= 0$, we could perform the following boost
\[
\begin{split}
y_0' =& \cosh \theta y_0 + \sinh\theta y_1 \\
y_1' =& \cosh \theta y_1 + \sinh\theta y_0.
\end{split}
\]
Under the boost, $ (A,(B,0,0),(0,C,0),(0,0,D))$ is transformed to 
\[   (A \cosh \theta ,(B, -D \sinh \theta,0),(-A \sinh \theta,C,0),(0,0,D  \cosh \theta )).  \]
We observe that under the boost, the value of 
\[  A^2 +  |\vec{D}|^2 - |\vec{B}|^2-|\vec{C}|^2 \]
is preserved. So is the condition $|A| > |\vec{D}|$.

On the other hand, starting with a general $(A, \vec{B},\vec{C},\vec{D})$, up to an $SO(3)$ action on $y^1,y^2,y^3$, we may assume
\[ K=(a,(b_1,b_2,0),(c_1,c_2,0),(0,0,d))  \]
where 
\[  a c_1 - b_2d = 0 . \]
This implies that the vector $(a,d)$ is parallel to $(c_1,b_2)$. Moreover, $|A| > |\vec{C}|  $  implies that $|a| > |c_1|$. As a result, we can find $A'$, $B'$, $C'$, $D'$ and $\theta$ such that 
 \[K= (A' \cosh \theta ,(B', -D' \sinh \theta,0),(-A' \sinh \theta,C',0),(0,0,D'  \cosh \theta )).  \]
We then perform the boost mentioned before. As a result, it suffices to study the case 
\[ K= (A, \vec{B},\vec{C},\vec{D})= (A,(B,0,0),(0,C,0),(0,0,D))  \]
where $AD + BC = 0$.

The conditions 
\[
\begin{split}
|A| > & \max \{ |\vec{B}|, |\vec{C}|, |\vec{D}| \} \\
A^2 +  |\vec{D}|^2 > & |\vec{B}|^2+|\vec{C}|^2
\end{split}
\]
is the same as 

\[
\begin{split}
|A| > & \max \{|B|,|C|, |D| \} \\
A^2 +  D^2 > & B^2+C^2.
\end{split}
\]
From  $AD + BC = 0$ and $|A| >  \max \{|B|,|C| \}$, we conclude that $ \min\{|B|,|C|\}>|D|$ and $|A|>|D|$. As mentioned before,  $A >  \max \{|B|,|C| \}$ is clearly a necessary condition for the Killing field to be everywhere timelike, in the following, we show that it is also a sufficient condition.

For a Killing field $K$ of the form \eqref{Killing},
we compute
\[
\begin{split}
-\langle K, K\rangle_{AdS^{3,1}}^2 =&(Ay^0+Cy^2)^2+ (By^1- Ay^4)^2 -(By^0- Dy^2)^2-(Cy^4+ Dy^1)^2\\
 =& (A^2-B^2)(y^0)^2 +(A^2-C^2)(y^4)^2+(C^2-D^2)(y^2)^2 + (B^2-D^2)(y^1)^2 \\
    & + 2(AC+BD)y^0y^2-2(AB+CD)y^1y^4\\
=& (\sqrt{A^2-B^2} y^0 + \sqrt{C^2-D^2} y^2)^2 +  (\sqrt{A^2-C^2} y^1 + \sqrt{B^2-D^2} y^4)^2.
 \end{split}
\]
In the last equality, we use the condition $|A| >  \max\{|B|,|C|\}$ and $ \min\{|B|,|C|\}>|D|$ to take the square root and the condition $AD+BC=0$ to complete the square. This shows that $K$ is everywhere timelike. 

To evaluate value of $\min -\langle K, K\rangle_{\R^{3,2}}$, we observe the following statement: Given two numbers $a$ and $b$ such that $a > |b|$ and fix $s$, we have
\[ \min_{x^2-y^2=s^2} (ax+by)^2 = (a^2-b^2)s^2 \]
which can be proved by either Lagrange multiplier or the parametrization $x=s\cosh \theta$ and $y=s \sinh \theta$. Finally, to minimize 
\begin{equation}\label{min_lagrange}
\begin{split}
-\langle K, K\rangle_{AdS^{3,1}}^2 =  (\sqrt{A^2-B^2} y^0 + \sqrt{C^2-D^2} y^2)^2 +  (\sqrt{A^2-C^2} y^1 + \sqrt{B^2-D^2} y^4)^2
 \end{split}
\end{equation}
under the constraint 
\[ (y^0)^2+ (y^4)^2 - (y^1)^2-(y^2)^2 = 1 +(y^3)^2 \ge 1,\]
if $ (y^0)^2 -(y^2)^2$ and $(y^4)^2 - (y^1)^2$ are both non-negative, then we apply \eqref{min_lagrange} to both terms and conclude that the minimal value is indeed $A^2+D^2-B^2-C^2$. Otherwise, we may assume that  $ (y^0)^2 -(y^2)^2 > 1$. We simply apply \eqref{min_lagrange} to the term $ (\sqrt{A^2-B^2} y^0 + \sqrt{C^2-D^2} y^2)^2 $ and conclude that the value has to be larger.

The second and simpler case is $\vec{D}=0$. We can use a coordinate boost to rewrite $K$ into the simpler form 
\[ K= (A, \vec{B},\vec{C},\vec{D})= (A,(B,0,0),(0,0,0),(0,0,0))  \]
 where equation \eqref{positive_condition} is the same as $|A|>|B|$. The rest of the argument is the same as before. 
\end{proof}

Combing the above two Lemmas, we obtain: 
\begin{proposition}\label{observer_constraint}  A Killing field $K$ of the form \eqref{Killing}  is an observer Killing field in $\mathcal{O}$ if and only if
\[
\begin{split}
A \vec{D} =  - \vec{B} \times \vec{C},A > & \max \{ |\vec{B}|, |\vec{C}|, |\vec{D}| \} 
\end{split}
\] and
\[ A^2 +  |\vec{D}|^2  -   |\vec{B}|^2-|\vec{C}|^2=1.\]
\end{proposition}

\section{Spinorial Killing fields of $AdS^{3,1}$}
 We first review the  Killing spinors and the construction of the corresponding Killing fields as follows. An irreducible representation of Clifford algebra of $\mathbb{R}^{3,1}$ on $\mathbb{C}^4$ is given by

\[ \gamma_0=\left[ \begin{array}{clclclc} 0 && 0 && 1 && 0 \\
                                          0 && 0 && 0 && 1 \\
                                          1 && 0 && 0 && 0 \\
                                          0 && 1 && 0 && 0
                   \end{array} \right],\ \gamma_1=\left[ \begin{array}{clclclc} 0 && 0 && 1 && 0 \\
                                          0 && 0 && 0 && -1 \\
                                          -1 && 0 && 0 && 0 \\
                                          0 && 1 && 0 && 0
                   \end{array} \right] \]
\[ \gamma_2=\left[ \begin{array}{clclclc} 0 && 0 && 0 && 1 \\
                                          0 && 0 && 1 && 0 \\
                                          0 && -1 && 0 && 0 \\
                                          -1 && 0 && 0 && 0
                   \end{array} \right],\ \gamma_3=\left[ \begin{array}{clclclc} 0 && 0 && 0 && -i \\
                                          0 && 0 && i && 0 \\
                                          0 && i && 0 && 0 \\
                                          -i && 0 && 0 && 0
                   \end{array} \right]. \]
Denote by $\gamma_{i_1i_2\dots i_k}=\gamma_{i_1}\gamma_{i_2}\dots \gamma_{i_k}$.  On $\mathbb{C}^4$ there is a form defined by
\[ ((v_1,v_2),(u_1,u_2))=\left\langle v_1,u_2 \right\rangle+\left\langle v_2,u_1 \right\rangle,\ \ v_j,u_j\in \mathbb{C}^2 \]
which is linear in the first argument $(v_1, v_2)$ and conjugate linear in the second one $(u_1, u_2)$. It can be easily checked that $(X.\psi,\phi)=(\psi,X.\phi)$ for any vector $X\in \R^{3,1}$, where $``."$ is the Clifford multiplication.  This form is $SL(2,\mathbb{C})$-invariant where the $SL(2,\mathbb{C})$ action is given by
\[  \rho (g)=\left[ \begin{array}{clc} \bar{g}&&0 \\
                                        0 && (g^t)^{-1}  
                   \end{array} \right],\] for any $g\in SL(2, \mathbb{C})$. 
                   
All these can be extended to any 4-dimensional spin Lorentzian manifold, where the spinor bundle is the complex rank 4 bundle associated with 
the standard representation.

 Let $\Sigma$ be the irreducible spinor bundle of $AdS^{3,1}$. 

\begin{definition}\label{Killing_spinor}
A section $\sigma$ of the complex rank 4 spinor bundle on $AdS^{3,1}$ is an imaginary Killing spinor if $\nabla_X\sigma=\frac{i}{2}X.\sigma$ where $\nabla$ is the spinor connection. 
\end{definition}

The following Lemma is well-known and the proof is included for completeness. 
\begin{lemma} \label{spinor_to_vector1}
Suppose $\psi$ is an imaginary Killing spinor on $AdS^{3,1}$, then the 1-form $\alpha$ defined by $\alpha(X)=-(X.\psi,\psi)$ is a Killing 1-form.
\end{lemma}
\begin{proof}
\begin{align*}
 \nabla \alpha (e_k,e_j)+\nabla \alpha (e_j,e_k)&=-(\frac{i}{2}e_j.e_k.\psi,\psi)-(e_j.\psi,\frac{i}{2}e_k.\psi)-(\frac{i}{2}e_k.e_j.\psi,\psi)-(e_k.\psi,\frac{i}{2}e_j.\psi)\\
 &=-\frac{i}{2}([e_j,e_k].\psi,\psi)-\frac{i}{2}([e_k,e_j].\psi,\psi)=0.
\end{align*}
\end{proof}

For a spinor Killing vector field (Definition \ref{spinor_Killing}) of $AdS^{3,1}$ of the form \eqref{Killing}, the coefficients $A, \vec{B}, \vec{C}, \vec{D}$ can
be explicitly related to imaginary Killing spinor $\psi$ as follows.

Take the base point $o\in AdS^{3,1}$ with $\R^{3,2}$ coordinates $y_0=1, y_i=0, i=1, 2, 3, y_4=0$. At $o$, the frame $\{\frac{\partial}{\partial y^4},\ \frac{\partial}{\partial y^1},\ \frac{\partial}{\partial y^2},\ \frac{\partial}{\partial y^3}\}$ is orthonormal.
We choose a trivialization of the spinor bundle $o$ such that the Clifford multiplications by $\frac{\partial}{\partial y^4}$ and $\frac{\partial}{\partial y^i}, i=1, 2, 3$ are identified with  matrix multiplication by $\gamma_0$ and $\gamma_i, i=1, 2, 3$, respectively.  The set of imaginary Killing spinors is parametrized by the fiber of the spinor bundle at $o$, which is isomorphic to $\mathbb{C}^4$.
\begin{lemma}
Let $\psi$ be an imaginary Killing spinor and $K$ be the corresponding Killing field of the form \eqref{Killing}. Then the coefficients of $K$ are:
\begin{equation}\label{coefficients}
\begin{aligned}
A=(\gamma_0\psi,\psi),&\ \   B_j=- (\gamma_j\psi,\psi)\\
C_j=(i\gamma_{0j}\psi,\psi) &,\ \ D_j= (\frac{i}{2}\epsilon^{jkl}\gamma_{kl}\psi,\psi)
\end{aligned}
\end{equation}
where all quantities above are evaluated at the base point $o$
\end{lemma}
\begin{proof}
Let $\alpha$ be the 1-form dual to $K$. With the above identification and \eqref{alpha}, $A=-\alpha(\frac{\partial}{\partial y^4})=(\frac{\partial}{\partial y^4}.\psi,\psi)$. 
$B_j=\alpha (\frac{\partial}{\partial y^j})=-(\frac{\partial}{\partial y^j}.\psi,\psi)$. By \eqref{d_alpha}, $C_j=\frac{1}{2}d\alpha (\frac{\partial}{\partial y^4},\frac{\partial}{\partial y^j})$ and $D_j=\frac{\epsilon^{jkl}}{4}d\alpha (\frac{\partial}{\partial y^k},\frac{\partial}{\partial y^l})$.

We compute 
\begin{align*}
 d\alpha (X,Y)&=-X(Y.\psi,\psi)+Y(X.\psi,\psi)+([X,Y].\psi,\psi)\\
              &=-(Y.\frac{i}{2}X.\psi,\psi)-(Y.\psi,\frac{i}{2}X.\psi)+(X.\frac{i}{2}Y.\psi,\psi)+(X.\psi,\frac{i}{2}Y.\psi)\\
              &=i([X,Y].\psi,\psi).
\end{align*} Therefore, 
\[ C_j=\frac{i}{2}([\frac{\partial}{\partial y^4},\frac{\partial}{\partial y^j}].\psi,\psi )=(i\frac{\partial}{\partial y^4}.\frac{\partial}{\partial y^j}.\psi,\psi ). \]
\[ D_j=\frac{i}{4}\epsilon^{jkl}([\frac{\partial}{\partial y^k},\frac{\partial}{\partial y^l}].\psi,\psi)=\frac{i}{2}\epsilon^{jkl}(\frac{\partial}{\partial y^k}.\frac{\partial}{\partial y^l}.\psi,\psi). \]
\end{proof}

Since the space of imaginary spinors can be identify with $\mathbb{C}^4$ and the space of Killing fields is identified with $\underline{so}(3,2)$, we obtain the following corollary:

\begin{corollary} \label{spinor_map}The map $\psi\mapsto (A, \vec{B}, \vec{C}, \vec{D})$ where $A, \vec{B}, \vec{C}, \vec{D}$ are given in \eqref{coefficients} can be considered as a map from 
$\mathbb{C}^4$ to $\underline{so}(3,2)$ whose image is the set $\mathcal{S}$ of spinor Killing fields.
\end{corollary}

To compare the sets $\mathcal{S}$ and $\mathcal{O}$, we will characterize the set $\mathcal{S}$  in terms of its defining equations.
A key to deriving these defining equations is the Fierz identity, which is briefly reviewed here. See \cite{VP} for more details. 

Let $M_{\mathbb{C}}(4)$ be the complex vector space of $4\times 4$ complex matrices and denote by $I^{\,\,\,\beta}_{\alpha}$ the identity matrix. The main idea is as follows: 
\[\{\frac{1}{2} I,\frac{1}{2} \gamma_0,\frac{1}{2} \gamma_j,\frac{1}{2} \gamma_{0j},\frac{1}{2} \gamma_{jk},\frac{1}{2} \gamma_{0jk},\frac{1}{2} \gamma_{123},\frac{1}{2} \gamma_{0123} \}\]
 form a unitary basis for $M_{\mathbb{C}}(4)$ with respect to the Hermitian product defined using the trace. Let $\mathfrak{B}$ denote the index set of $(a_1, \cdots, a_k), k=0, 1, 2, 3, 4$ in this basis. As a result, any matrix can be expressed  as a linear combination of the basis where the coefficients can be computed effectively. A direct calculation shows that  each $\gamma_{a_1,\dots a_k}$ is either Hermitian or skew-Hermitian. Precisely, the complex conjugate $\gamma^*_{a_1\dots a_k}$ is

\begin{align*}
\gamma^*_{a_1\dots a_k} &=(-1)^{k(k+1)/2+s}\gamma_{a_1\dots a_k}, \text{   where  }
s&= \left\{ \begin{array}{clc} 1 && \textup{if one of }\ a_1\dots a_k\ \textup{is}\ 0 \\
                       0 && \textup{otherwise} \end{array}\right . .
\end{align*}

Any matrix $M_{\alpha} ^{\ \ \beta}$ in $M_{\mathbb{C}}(4)$ can be expressed as 
\begin{equation}\label{Fierz-matrix}
 4M_{\alpha}^{\ \ \beta} = \sum_{(a_1,\cdots, a_k)\in \mathfrak{B}}(-1)^{k(k+1)/2+s} (\gamma_{a_1\dots a_k})_{\alpha}^{\ \ \beta} Tr( \gamma_{a_1\dots a_k}M).
\end{equation}
Since equation \eqref{Fierz-matrix} holds for any matrix $M$, it follows that
\begin{equation}\label{Fierz}
4I_{\alpha}^{\ \ \mu}I_{\nu}^{\ \ \beta} = \sum_{(a_1, \cdots, a_k) \in \mathfrak{B}}(-1)^{k(k+1)/2+s} (\gamma_{a_1\dots a_k})_{\alpha}^{\ \ \beta} (\gamma_{a_1\dots a_k})_{\nu}^{\ \ \mu}.
\end{equation}

\begin{proposition}
A Killing field $K$ of the form \eqref{Killing}  is a spinor Killing field in $\mathcal{S}$ if and only if $(A, \vec{B}, \vec{C}, \vec{D})$ with $A>0$ satisfies the following relations:
\begin{equation}\label{def_V}\begin{split} A^2&=|\vec{B}|^2+|\vec{C}|^2+|\vec{D}|^2-2\Delta^2\\
(\vec{B}\cdot \vec{C})^2&=(|\vec{B}|^2-\Delta^2)(|\vec{C}|^2-\Delta^2)\\
(\vec{C}\cdot \vec{D})^2&=(|\vec{C}|^2-\Delta^2)(|\vec{D}|^2-\Delta^2)\\
(\vec{B}\cdot \vec{D})^2&=(|\vec{B}|^2-\Delta^2)(|\vec{D}|^2-\Delta^2)\\
&|\vec{B}|^2,  |\vec{C}|^2,\ |\vec{D}|^2\geq\Delta^2\geq 0\\
&(\vec{B}\cdot\vec{C})(\vec{C}\cdot\vec{D})(\vec{D}\cdot\vec{B})\geq 0
\end{split}\end{equation} where  $\Delta^2=-\vec{B}\cdot (\vec{C}\times \vec{D})/A$.

\end{proposition}
\begin{proof} Let $\mathcal{V}$ be the set of $(A, \vec{B}, \vec{C}, \vec{D})$ that satisfies \eqref{def_V}. We prove that $\mathcal{S}=\mathcal{V}$
in the following. 
\bigskip

\noindent {\bf Part 1: $\mathcal{S}\subset\mathcal{V}$}

Let $K\in \mathcal{S} $ be a spinor Killing field which is defined by $\psi\in\mathbb{C}^4$ through the identification at the base point $o$. Define $A,B_j,C_j,D_j$ as before. Furthermore, we define
\begin{align*}
E_j=\left\langle\frac{i}{2}\epsilon^{jkl}\gamma_{kl}\psi,\psi\right\rangle &,\ \ G=\left\langle\gamma_0\psi,\psi \right\rangle\\
F=-\left\langle\gamma_{123}\psi,\psi \right\rangle &,\ \ H=-\left\langle i\gamma_{0123}\psi,\psi \right\rangle.
\end{align*}
We derive relations among $A, \vec{B}, \vec{C}, \vec{D}, \vec{E}, F, G, H$  by contracting the Fierz identity with various spinors.

Contracting (\ref{Fierz}) with $\psi^\alpha\bar{\psi}_\beta\psi^\nu\bar{\psi}_\mu$ gives
\[ 3A^2=|\vec{B}|^2+|\vec{C}|^2+|\vec{D}|^2+|\vec{E}|^2+F^2+G^2+H^2. \]
Contracting (\ref{Fierz}) with $(\gamma_{0j})^\alpha_{\ \ \tilde{\alpha}}\psi^{\tilde{\alpha}}\bar{\psi}_\beta\psi^\nu\bar{\psi}_\mu$ gives
\[ A\vec{B}=- \vec{C}\times\vec{D}+H\vec{E}. \]
Contracting (\ref{Fierz}) with $(\gamma_{j})^\alpha_{\ \ \tilde{\alpha}}\psi^{\tilde{\alpha}}\bar{\psi}_\beta\psi^\nu\bar{\psi}_\mu$ gives
\[ A\vec{C}=-\vec{D}\times\vec{B}+F\vec{E}. \]
Contracting (\ref{Fierz}) with $(\gamma_{0ij})^\alpha_{\ \ \tilde{\alpha}}\psi^{\tilde{\alpha}}\bar{\psi}_\beta\psi^\nu\bar{\psi}_\mu$ gives
\[ A\vec{D}=-\vec{B}\times\vec{C}+G\vec{E}. \]
Contracting (\ref{Fierz}) with $(\gamma_{0123})^\alpha_{\ \ \tilde{\alpha}}\psi^{\tilde{\alpha}}\bar{\psi}_\beta\psi^\nu\bar{\psi}_\mu$ gives
\[ AH= \vec{B}\cdot \vec{E}. \]
Contracting (\ref{Fierz}) with $(\gamma_{123})^\alpha_{\ \ \tilde{\alpha}}\psi^{\tilde{\alpha}}\bar{\psi}_\beta\psi^\nu\bar{\psi}_\mu$  gives
\[ AF= \vec{C} \cdot \vec{E}. \]
Contracting (\ref{Fierz}) with $(\gamma_{0})^\alpha_{\ \ \tilde{\alpha}}\psi^{\tilde{\alpha}}\bar{\psi}_\beta\psi^\nu\bar{\psi}_\mu$ gives
\[ AG=\vec{D}\cdot \vec{E}. \]
Contracting (\ref{Fierz}) with $(\gamma_{ij})^\alpha_{\ \ \tilde{\alpha}}\psi^{\tilde{\alpha}}\bar{\psi}_\beta\psi^\nu\bar{\psi}_\mu$ gives
\[ A\vec{E}=H\vec{B}+F\vec{C}+G\vec{D}. \]
In conclusion, we have
\begin{equation}\label{system_unsimplified}
\begin{split}
3A^2&=|\vec{B}|^2+|\vec{C}|^2+|\vec{D}|^2+|\vec{E}|^2+F^2+G^2+H^2\\
A\vec{B}&=-\vec{C}\times\vec{D}+H\vec{E}\ \ \ \ \ AH=\vec{B}\cdot \vec{E}  \\
A\vec{C}&=-\vec{D}\times\vec{B}+F\vec{E}\ \ \ \ \ AF=\vec{C}\cdot \vec{E}  \\
A\vec{D}&=-\vec{B}\times\vec{C}+G\vec{E}\ \ \ \ \ AG=\vec{D}\cdot  \vec{E}\\
A\vec{E}&=H\vec{B}+F\vec{C}+G\vec{D}.
\end{split}
\end{equation}

In the following calculations, we eliminate $\vec{E}, F, G, H$ from equation \eqref{system_unsimplified} to derive equation \eqref{def_V}. 

\begin{lemma}
Equation \eqref{system_unsimplified} implies
\begin{equation}
\begin{split}
H^2&=|\vec{B}|^2-\Delta^2\\
F^2&=|\vec{C}|^2-\Delta^2\\
G^2&=|\vec{D}|^2-\Delta^2.
\end{split}
\end{equation}
\end{lemma}
\begin{proof}
Starting with
\[
 A\vec{B} =-\vec{C}\times\vec{D}+H\vec{E}, 
\]
we take the inner product of both sides with $ \vec{B}$. We have
\[
 A\vec{B} \cdot \vec{B} =(-\vec{C}\times\vec{D}+H\vec{E} ) \cdot \vec{B}.
\]
Using  $AH=\vec{B}\cdot \vec{E}$ and the definition of $\Delta$, we conclude
\[H^2=|\vec{B}|^2-\Delta^2.  \]
The other two equations can be derived similarly.
\end{proof}
In particular, this implies 
\[ |\vec{B}|^2,  |\vec{C}|^2,\ |\vec{D}|^2\geq\Delta^2\geq 0. \]
Next we prove
\begin{lemma}
Equation \eqref{system_unsimplified} implies
\begin{equation}
\begin{split}
\vec{B}\cdot\vec{D}= & HG\\
\vec{C}\cdot\vec{D}=&GF\\
\vec{C}\cdot\vec{B}=&FH.
\end{split}
\end{equation}
\end{lemma}
\begin{proof} Starting with 
\[ A\vec{B} =-\vec{C}\times\vec{D}+H\vec{E}, \]
we take the inner product of both sides with $\vec{D}$ and use
$AG=\vec{D}\cdot \vec{E}$. We conclude that
\[ \vec{B}\cdot\vec{D}=  HG.  \]
The other two equations can be derived similarly.
\end{proof}
The second through the forth equations in equation \eqref{def_V} follows from the above two lemmas. For example,
\[
\begin{split}
(\vec{B}\cdot \vec{C})^2 = & F^2H^2 \\
=&(|\vec{B}|^2-\Delta^2)(|\vec{C}|^2-\Delta^2).
\end{split}
\]

Finally, we show that the first equation in \eqref{def_V} follows from equation \eqref{system_unsimplified}.
We start with $AF=\vec{C}\cdot\vec{E}$ and use also $ A\vec{E}=H\vec{B}+F\vec{C}+G\vec{D}$. We compute
 \begin{align*}
AF&=\vec{C}\cdot\vec{E}\\
&=\frac{1}{A}\left( F|\vec{C}|^2+G\vec{C}\cdot\vec{D}+H\vec{B}\cdot\vec{C} \right)\\
 &=\frac{F}{A}\left(F^2+G^2+H^2+\Delta^2 \right).
\end{align*}
As a result,
 \[  A^2=F^2+G^2+H^2+\Delta^2=|\vec{B}|^2+|\vec{C}|^2+|\vec{D}|^2-2\Delta^2. \] 

This finishes the proof of this part.
\bigskip

\noindent {\bf Part 2: $\mathcal{V}\subset\mathcal{S}$}

The proof of this direction consists of four lemmas.
\begin{lemma}\label{permutation}
$(A,\vec{B},\vec{C},\vec{D})\in \mathcal{V}$ if and only if $(A,\vec{C},\vec{D},\vec{B})\in \mathcal{V}$. 
$(A,\vec{B},\vec{C},\vec{D})\in \mathcal{S}$ if and only if $(A,\vec{C},\vec{D},\vec{B})\in \mathcal{S}$ 
\end{lemma}
\begin{proof}
The first statement is obvious in the view of \eqref{def_V}. For the second one, suppose $\psi=(w_1,w_2,u_1,u_2)\in\mathbb{C}^4$. Let $\tilde{\psi}=(iw_1+iu_1,iw_2+iu_2,w_1-u_1,w_2-u_2)/ \sqrt{2}$. Then by direct computation we have 
\begin{align*}
A(\psi)=A(\tilde{\psi}),&\ \ \vec{B}(\psi)=\vec{C}(\tilde{\psi})\\
\vec{C}(\psi)=\vec{D}(\tilde{\psi}),&\ \ \vec{D}(\psi)=\vec{B}(\tilde{\psi}).
\end{align*}
\end{proof} 
\begin{lemma}\label{rotation} We have the following two statements:

\begin{enumerate}
\item $(A,\vec{B},\vec{C},\vec{D})\in \mathcal{S}$ if and only if 
$(A,\cos \theta\vec{B}+\sin \theta\vec{C},-\sin\theta\vec{B}+\cos\theta\vec{C},\vec{D})\in \mathcal{S}$.
\item $(A,\vec{B},\vec{C},\vec{D})\in \mathcal{V}$ if and only if 
$(A,\cos \theta\vec{B}+\sin \theta\vec{C},-\sin\theta\vec{B}+\cos\theta\vec{C},\vec{D})\in \mathcal{V}$.
\end{enumerate}
\end{lemma}
\begin{proof}
The first statement follows that $\mathcal{S}$ is invariant under the isometry group of $AdS^{3,1}$ and that this transformation corresponds to the time translation. We begin to prove the second one. Since $\Delta^2,\ |\vec{B}|^2+\ |\vec{C}|^2$ and $|\vec{B}\times\vec{C} |^2$ are unchanged under the transformation, it's sufficient to show $(\vec{B}\cdot\vec{D})^2=(|\vec{B}|^2-\Delta^2)(|\vec{D}|^2-\Delta^2)$ is preserved.

We compute
\begin{align*}
&\left((\cos \theta\vec{B}+\sin \theta\vec{C})\cdot\vec{D}\right)^2\\
=&\cos^2\theta (\vec{B}\cdot\vec{D})^2+\sin^2\theta (\vec{C}\cdot\vec{D})^2+2\sin\theta\cos\theta(\vec{B}\cdot\vec{D})(\vec{C}\cdot\vec{D})
\end{align*}
and
\begin{align*}
&\left( |\cos\theta\vec{B}+\sin\theta\vec{C}|^2-\Delta^2 \right)(|\vec{D}|^2-\Delta^2)\\
=&\left(\cos^2\theta (|\vec{B}|^2-\Delta^2)+\sin^2\theta (|\vec{C}|^2-\Delta^2)+2\cos\theta\sin\theta\vec{B}\cdot\vec{C} \right)(|\vec{D}|^2-\Delta^2)\\
=&\cos^2\theta (\vec{B}\cdot\vec{D})^2+\sin^2\theta (\vec{C}\cdot\vec{D})^2+2\cos\theta\sin\theta\vec{B}\cdot\vec{C}(|\vec{D}|^2-\Delta^2).
\end{align*}
The lemma follows since $(\vec{B}\cdot\vec{D})^2(\vec{C}\cdot\vec{D})^2=(\vec{B}\cdot\vec{C})^2(|\vec{D}|^2-\Delta^2)^2$ and $(\vec{B}\cdot\vec{D})(\vec{C}\cdot\vec{D}),\ (\vec{B}\cdot\vec{C})$ have the same sign.
\end{proof} 
\begin{lemma}\label{zero}
Suppose $(A,\vec{B},\vec{C},\vec{D})\in\mathcal{V}$ and $\vec{B}=0$, then $(A,\vec{B},\vec{C},\vec{D})\in\mathcal{S}$.
\end{lemma}
\begin{proof}
We have $\Delta^2=0$,\ $A^2=|\vec{C}|^2+|\vec{D}|^2$. This implies that $\vec{C},\ \vec{D}$ are parallel. Picking $\psi=(1,\exp(-i\theta),\exp(i\theta),1)/2$,
\begin{align*}
&A(\psi),\vec{B}(\psi),\vec{C}(\psi),\vec{D}(\psi))\\
=&(1,(0,0,0),(0,-\cos\theta\sin\theta,\sin^2\theta),(0,\cos^2\theta,-\cos\theta\sin\theta)).
\end{align*}
Thus $(A,\vec{B},\vec{C},\vec{D})$ can be achieved by choosing suitable $\theta$ and applying Lemma \ref{rotation} if necessary.  
\end{proof}
\begin{lemma}\label{reduction}
If $(A_0,\vec{B_0},\vec{C_0},\vec{D_0})\in \mathcal{V}$ satisfies $|\vec{B}_0|^2=\Delta_0^2$, then it is also in $\mathcal{S}$.
\end{lemma}
\begin{proof}
In view of \eqref{def_V}, we know $\vec{B}_0\cdot\vec{C}_0=\vec{B}_0\cdot\vec{D}_0=0$. By dilation, rotation and Lemma \ref{rotation} we can assume 
\[(A_0,\vec{B_0},\vec{C_0},\vec{D_0})=(1,(-b,0,0),(0,c,-\gamma),(0,\mu,c)).\]
\eqref{def_V} becomes
\begin{align*}
1+b^2&=\mu^2+\gamma^2+2c^2\\
(\gamma-\mu)^2c^2&=(\mu^2+c^2-b^2)(\gamma^2+c^2-b^2)\\
b(\mu\gamma+c^2)&=b^2.
\end{align*}
From Lemma \ref{zero} we can assume $b\neq 0$. Eliminating $c^2$ in the first and the third lines we obtain a quadratic equation of $b$
\begin{align*}
1+b^2&=\mu^2+\gamma^2+2b-2\mu\gamma.
\end{align*}
We conclude $b=1\pm (\mu-\gamma)$ and $c^2=b-\mu\gamma=1-\mu\gamma\pm (\mu-\gamma)$.\\\\

\textbf{Case 1}: $\mu=\gamma$\\
In this case $\vec{C}_0,\ \vec{D}_0$ are orthogonal and $|\vec{B}_0|=|\vec{C}_0|=|\vec{D}_0|=1$. Pick $\psi=(-1,i,,i,-1)/2$ then
\[(A,\vec{B},\vec{C},\vec{D})=(1,(0,0,-1),(1,0,0),(0,1,0)).\]
Together with Lemma \ref{permutation} and Lemma \ref{rotation}, this situation can be obtained.\\\\
\textbf{Case 2}:\ $\mu < \gamma$\\
In this case from $b\leq 1$ we have $b=1+(\mu-\gamma)$. Let 
\begin{align*}
\alpha=\frac{\mu+\gamma}{\mu-\gamma},\ \beta=-\frac{2\textup{sgn}(c)}{\gamma-\mu}\sqrt{(1-\gamma)(1+\mu)}.
\end{align*}
Then $\psi=\frac{\sqrt{\gamma-\mu}}{2}(1,w_2,-\bar{w_2},-1)$ with $w_2^2=\alpha+i\beta$ has the desired image.\\\\
\textbf {Case 3}:\ $\mu> \gamma$\\
In this case form $b\leq 1$ we have $b=1- (\mu-\gamma)$. Assign
\begin{align*}
\alpha=\frac{\mu+\gamma}{\mu-\gamma},\ \ \beta=\frac{2\textup{sgn}(c)}{\mu-\gamma}\sqrt{(1+\gamma)(1-\mu)}
\end{align*}
and $\psi=\frac{\sqrt{\mu-\gamma}}{2}(1,w_2,\bar{w_2},1)$ with $w_2^2=\alpha+i\beta$ has the desired image.
\end{proof} 
Now we can show $\mathcal{V}\subset\mathcal{S}$. Given $(A,\vec{B},\vec{C},\vec{D})\in \mathcal{V}$. By dilation we can assume $A=1$. By Lemma  \ref{rotation} we can replace $\vec{B},\vec{C}$ by $\cos\theta\vec{B}+\sin\theta\vec{C},-\sin\theta\vec{B}+\cos\theta\vec{C}$. By choosing $\theta$ such that $\cos\theta (\vec{B}\cdot\vec{D})+\sin\theta (\vec{C}\cdot\vec{D})=0$, we can assume $\vec{B}\cdot\vec{D}=0$. From (\ref{def_V}) this implies either $|\vec{B}|^2=\Delta^2$ or $|\vec{D}|^2=\Delta^2$. From Lemma \ref{permutation} we can assume $|\vec{B}|^2=\Delta^2$. And it is done in Lemma \ref{reduction}.
\end{proof}

\section{Spinor Killing fields vs. Observer Killing fields}
We study the relation between the two subsets $\mathcal{O}$ and $\mathcal{S}$ of Killing fields in this section. The following lemma for the subset
$\mathcal{O}$ is needed in the proof of Theorem 

\begin{lemma}\label{causal}
For any $(A,\vec{B},\vec{C},\vec{D})\in\mathcal{O}$ we have
\[  \min_{|\vec{x}|=1} A^2 -(\vec{B}\cdot\vec{x})^2-(\vec{C}\cdot\vec{x})^2-(\vec{D}\cdot\vec{x})^2\geq 0. \]
The equality holds if and only if $(A,\vec{B},\vec{C},\vec{D})$ is null. And in this case, the minimum is attained only if $\vec{D}\cdot\vec{x}=0$.
\end{lemma}

\begin{proof}
Up to rotation and dilation, we may assume 
\[(A,\vec{B},\vec{C},\vec{D})=(1,(b,0,0),(c_1,c_2,0),(0,0,d)).\]
Let $\vec{x}=(\alpha,\beta,\gamma)$, then
\begin{align*}
(\vec{B}\cdot\vec{x})^2+(\vec{C}\cdot\vec{x})^2+(\vec{D}\cdot\vec{x})^2&=(\alpha b)^2+(\alpha c_1+\beta c_2)^2+(\gamma d)^2\\
&=(b^2+c_1^2-d^2)\alpha^2+(c_2^2-d^2)\beta^2+2c_1c_2\alpha\beta+d^2.
\end{align*}
In view of $|\vec{D}|\leq |\vec{B}||\vec{C}|$ and $|\vec{B}|,|\vec{C}|\leq 1$ we have $|\vec{D}|\leq |\vec{B}|,\ |\vec{C}|$. Then the above expression attends its maximum only if $\gamma=0$. We can denote $\alpha=\cos\theta,\ \beta=\sin\theta$. Then above expression becomes
\begin{align*}
&(b^2+c_1^2-d^2)\cos^2\theta+(c_2^2-d^2)\sin^2\theta+2c_1c_2\cos\theta\sin\theta+d^2\\
=&\left(\frac{b^2+c_1^2-c_2^2}{2}\right)\cos 2\theta+c_1c_2\sin 2\theta+\frac{b^2+c_1^2+c_2^2}{2}\\
\leq &\sqrt{\left(\frac{b^2+c_1^2-c_2^2}{2}\right)^2+c_1^2c_2^2}+\frac{b^2+c_1^2+c_2^2}{2}\\
\leq &1.
\end{align*}
The last inequality comes from $1+d^2\geq b^2+c_1^2+c_2^2$ and $d=bc_2$.
\end{proof}

We prove Theorem \ref{convex} which gives the precise relation between the sets  $ \mathcal{O}$ and $\mathcal{S}$.

\bigskip
\noindent {\bf Theorem \ref{convex} } {\it
The convex hull of  $ \mathcal{O}$ is properly contained in that of $\mathcal{S}$.}

\bigskip
\begin{proof} We show that each  $y \in \mathcal{O}$ can be written as a linear combination of elements in $ \mathcal{S}$ with positive coefficients. 
We look at two cases depending on whether $\vec{D}$ vanishes. 

If $\vec{D}=0$ then we may assume that $y\in \mathcal{O}$ is of the form $(A, b,0,0, c, 0, 0, 0, 0, 0)$ where $A>0 $ and $A^2=b^2+c^2+1$. However, this is clearly in the convex null of   $\mathcal{S}$ since  $\mathcal{S}$ contains all vector of $(A, b,0,0, c, 0, 0, 0, 0, 0)$ where $A>0 $ and $A^2=b^2+c^2$.

If $\vec{D}$ is not zero, we may assume that $y\in \mathcal{O}$ is of the form $(A, 1,0,0, c_1, c_2, 0, 0, 0, d)$ where $Ad=-c_2$,
\[  A^2+d^2 -1-c_1^2-c_2^2 = 1.  \]
and $A > 1> c_1^2+c_2^2$. We may also assume $1>d>0$.

We observe that $(A, \vec{0}, \vec{0}, \vec{D})\in \mathcal{S}$ if $A=|\vec{D}|$.  As a result, any $(A, \vec{0}, \vec{0}, \vec{D})$ with $A > |\vec{D}|$ is in the convex hull of $ \mathcal{S}$. It suffices to find $\alpha > 0 $ and $|e|<1$ such that $(A-\alpha, 1,0,0, c_1, c_2, 0, 0, 0, d- \alpha e)$ is in $ \mathcal{S}$. 

Note that for $\vec{D}$ perpendicular to both $\vec{B}$ and $\vec{C}$, the defining equations for  $ \mathcal{S}$ reduces to 
\[\begin{split} A^2&=|\vec{B}|^2+|\vec{C}|^2 -|\vec{D}|^2\\
(\vec{B}\cdot \vec{C})^2&=(|\vec{B}|^2-\Delta^2)(|\vec{C}|^2-\Delta^2)\\
|\vec{D}|^2 & = -\vec{B}\cdot (\vec{C}\times \vec{D})/A.
\end{split}\]
In terms of $A$, $D$, $c_i$, $e$ and $\alpha$, the defining equations reads
\[\begin{split}
(\frac{c_2}{d- \alpha e})^2 = 1+c_1^2+c_2^2 + (d- \alpha e) ^2\\
d-\alpha e = - \frac{c_2}{A- \alpha}.
\end{split}\]
We only have two remaining equations since the first two equations both reduce to the first one. If we plug in $d= -\frac{c_2}{A}$ to the second equation, we have
\begin{align}
\label{adjust1}(\frac{c_2}{d- \alpha e})^2 = 1+c_1^2+c_2^2 + (d- \alpha e) ^2\\
\label{adjust2} \alpha e = \frac{c_2 \alpha}{A(A- \alpha)}.
\end{align}
We solve $ \alpha e $ from equation \eqref{adjust1} and then solve $\alpha$ from equation \eqref{adjust2}. We have to make sure that the solution exist  and the obtained solution satisfy $|\alpha e| < \alpha$ and $A - \alpha > 0$. 

Equation \eqref{adjust1} can be solved by applying intermediate value theorem. Consider the function $f$ 
\[  f(x) =(\frac{c_2}{d- x})^2 - 1-c_1^2-c_2^2 - (d- x) ^2.  \]
We have $f(0)=1> 0$ and $f(d-1)= -2 -c_1^2 < 0$. Hence, there is a solution of  equation \eqref{adjust1} $\alpha e$ between $d-1$ and $0$. With this solution, we can solve from equation \eqref{adjust2} that 
\[  \alpha = \frac{A^2 \alpha e}{A \alpha e + c_2}.\]
The denominator is non-zero since both $\alpha e$ and $c_2$ are negative. It is also clear from the expression that $\alpha$ is positive. Moreover, from
\[ A- \alpha = \frac{c_2 \alpha}{A \alpha e}, \]
we conclude that $A-\alpha >0$. 

Finally, 
\[  \frac{|\alpha e|}{\alpha} = \frac{1}{A} \frac{|c_2|}{A-\alpha}<1 \]
since $A >1$ and 
\[ (\frac{c_2}{A-\alpha})^2 =1+c_1^2+c_2^2 + (d- \alpha e) ^2 > 1. \]

This shows that the convex hull of  $\mathcal{O}$ is contained in that of  $\mathcal{S}$. To show that it is properly contained, it suffices to find a vector in  $\mathcal{S}$ which does not lie in the convex hull of  $\mathcal{O}$.  Let $\psi=(1,0,1+i,0)\in \mathbb{C}^4$ be the value of an imaginary Killing spinor at the base point $o$, the corresponding Killing field is
\[(A_0,\vec{B}_0,\vec{C}_0,\vec{D}_0)=(3,(-1,0,0),(-2,0,0),(2,0,0)).\]
For $\vec{x}_0=(1,0,0)$, we have
\[A_0^2= (\vec{B_0}\cdot\vec{x_0})^2+(\vec{C}_0\cdot\vec{x}_0)^2+(\vec{D}_0\cdot\vec{x}_0)^2. \]
Suppose $(A_0,\vec{B}_0,\vec{C}_0,\vec{D}_0)$ is contained in the convex hull of $\mathcal{O}$.
\[(A_0,\vec{B}_0,\vec{C}_0,\vec{D}_0)=\sum_{j}\lambda_j (A_j,\vec{B}_j,\vec{C}_j,\vec{D}_j),\ \lambda_j> 0,\ (A_j,\vec{B}_j,\vec{C}_j,\vec{D}_j)\in\mathcal{O} .\]
From Lemma \ref{causal}, we have $(A_j,\vec{B}_j\cdot\vec{x}_0,\vec{C}_j\cdot\vec{x}_0,\vec{D}_j\cdot\vec{x}_0)$ are future causal but $(A_0,\vec{B}_0\cdot\vec{x}_0,\vec{C}_0\cdot\vec{x}_0,\vec{D}_0\cdot\vec{x}_0)$ is null. Thus $(A_j,\vec{B}_j\cdot\vec{x}_0,\vec{C}_j\cdot\vec{x}_0,\vec{D}_j\cdot\vec{x}_0)$ are also null and $\vec{D}_j\cdot\vec{x}_0=0$ from Lemma \ref{causal}. But $\vec{D}_0\cdot\vec{x}_0=2\neq 0$. It's a contradiction. This finishes the proof of the theorem.
\end{proof}


\section{The observer energy and the rest mass of an asymptotically AdS spacetime.}

For an asymptotically AdS initial data, let $e$ be the total energy and $\vec{p}$, $\vec{c}$ and $\vec{j}$ be the total linear momentum, center of mass and angular momentum, respectively. (See  \cite{Chrusciel-Maerten-Tod} for the details on the asymptotic assumptions and the definition of the total conserved quantities.) These quantities are associated with the Killing fields of the background AdS spacetime. Precisely
\begin{align*}
y^0\frac{\partial}{\partial y^4}-y^4\frac{\partial}{\partial y^0}\ \longmapsto\  e\ ,&\ y^0\frac{\partial}{\partial y^k}+y^k\frac{\partial}{\partial y^0}\ \longmapsto\ p_k \\
y^4\frac{\partial}{\partial y^k}+y^k\frac{\partial}{\partial y^4}\ \longmapsto\ c_k,&\ \ \ \ \ \ \ \ \epsilon_{kqr} y^q\frac{\partial}{\partial y^r}\ \longmapsto\ j_k .
\end{align*} 
  The following positive energy theorem is proved in  \cite{Chrusciel-Maerten-Tod}.
\begin{theorem}\label{rigidity_Chrusciel}
For an asymptotically AdS initial data satisfying the dominant energy condition with conserved quantities $(e, \vec{p}, \vec{c}, \vec{j})$ and any spinor Killing vector field $K=(A, \vec{B}, \vec{C}, \vec{D})$ in $\mathcal{S}$,
\[  Ae+ \vec{B} \cdot \vec{p} + \vec{C} \cdot \vec{c} + \vec{D} \cdot \vec{j} \ge 0. \]
In particular, this implies $e \ge |\vec{p}|$ and as $\vec{p}=0$
\begin{equation} 
e \geq \sqrt{|\vec{c}|^2+ |\vec{j}|^2+2|\vec{c} \times \vec{j}|}.  
\end{equation}
Moreover,  the data comes from a hypersurface in AdS if any of the followings holds.
\begin{enumerate}
\item There are two linearly independent imaginary Killing spinors along the initial data.
\item $(e,\vec{p})$ is null.
\item The equality in equation \eqref{rigidity_Chrusciel} holds and  $\vec{c} \times \vec{j} =0$.
\end{enumerate}
\end{theorem}
\begin{remark}
In \cite{Chrusciel-Maerten-Tod}, two additional rigidity conditions are provided. For simplicity, we only state the ones that will be used later in this article.
\end{remark}
As an immediate consequence of Theorem \ref{convex}, 
\begin{corollary}
Let $(M,g,k)$ be an asymptotically AdS initial data satisfying the dominant energy condition and $e$, $\vec{p}$, $\vec{c}$ and $\vec{j}$  be the total energy, linear momentum, center of mass, and angular momentum as defined in \cite{Chrusciel-Maerten-Tod}. Then for any observer Killing field $K=(A,\vec{B},\vec{C},\vec{D})\in \mathcal{O}$,
\[  Ae+ \vec{B} \cdot \vec{p} + \vec{C} \cdot \vec{c} + \vec{D} \cdot \vec{j} \ge 0 \]
\end{corollary}

From the above positivity, we define the rest mass of an asymptotically AdS initial data as follows:
\begin{definition}\label{rest_mass}
The rest mass $m$ of an asymptotically AdS initial data  with conserved quantities $(e, \vec{p}, \vec{c}, \vec{j})$ is defined to be
\[ m= \inf_{(A, \vec{B}, \vec{C}, \vec{D}) \in \mathcal{O}} Ae+ \vec{B} \cdot \vec{p} + \vec{C} \cdot \vec{c} + \vec{D} \cdot \vec{j}, \] where $\mathcal{O}$ is the set of observer Killing fields.
\end{definition}

The conserved quantities $(e,\vec{p},\vec{c},\vec{j})$ can be considered as an element in $\underline{so}(3,2)^*$, the dual space of the Lie algebra $\underline{so}(3,2)$. $SO(3,2)$ is a semi-simple Lie group and the Killing form is non-degenerate. In terms of the coordinate we used, the Killing form is $B(K,K)=6(-A^2-|\vec{D}|^2+|\vec{B}|^2+|\vec{C}|^2)$. Thus $\frac{1}{6}B(.,.)$ induces an isomorphism $\underline{so}(3,2)^*\cong\underline{so}(3,2)$. For an element $K^*\in \underline{so}(3,2)^*$, denote \[\alpha(K^*)=-\frac{1}{6}\textup{Tr}(ad_{K^*} \mathord{\cdot} ad_{K^*})\] and \[\beta(K^*)=-\frac{1}{3}\textup{Tr}(ad_{K^*}\mathord{\cdot} ad_{K^*}\mathord{\cdot} ad_{K^*}\mathord{\cdot} ad_{K^*})+\frac{1}{12}\left(\textup{Tr}(ad_{K^*}\mathord{\cdot} ad_{K^*})\right)^2.\] 

\begin{remark}
One observes that the Killing form $B$ is related to the timelike condition in Lemma \ref{lemma_timelike} for Killing vector fields, and to the positivity result stated in the following theorem for conserved quantities.
\end{remark}
\begin{remark}
Different notions of the rest mass were proposed by other authors (such as \cite{Wang-Xie-Zhang}). They were  proved to be positive. However, it is not clear that these definitions are physically relevant. 
\end{remark}
\begin{theorem}
For $(e, \vec{p}, \vec{c}, \vec{j})$ arising from  an asymptotically AdS initial data satisfying the dominant energy condition, denote
\begin{equation}\label{alpha_beta}
\begin{aligned}
\alpha&=e^2+|\vec{j}|^2-|\vec{p}|^2-|\vec{c}|^2\\
\beta&=(e^2-|\vec{j}|^2-|\vec{p}|^2-|\vec{c}|^2)^2-4|\vec{j}\times\vec{p}|^2-4|\vec{p}\times\vec{c}|^2-4|\vec{c}\times\vec{j}|^2+8e\vec{c}\cdot (\vec{p}\times\vec{j}).
\end{aligned}
\end{equation}
Then $\alpha$ and $\beta$ are both non-negative and the rest mass $m$ of $(e, \vec{p}, \vec{c}, \vec{j})$ is of the form
\begin{align}
m^2&=\frac{1}{2}(\alpha+\sqrt{\beta}).
\end{align}
\end{theorem}
\begin{proof} We divide the proof into two cases.
\bigskip

\textbf{Case 1:} $\vec{p}=0$.\\
From Theorem \ref{rigidity_Chrusciel}, $e\geq\sqrt{|\vec{j}|^2+|\vec{c}|^2+2|\vec{j}\times\vec{c}|}$. In particular, this implies that both $\alpha$ and $\beta$ are non-negative. 

We first further assume $e > \sqrt{|\vec{j}|^2+|\vec{c}|^2+2|\vec{j}\times\vec{c}|}$. Let $K=(A, \vec{B}, \vec{C}, \vec{D})$ be an observer Killing field. Let $|\vec{C}|=\eta,\ |\vec{D}|=d$, and $\theta$ be the angle between $\vec{B}$ and $\vec{C}$. Proposition \ref{observer_constraint} implies
\begin{align*}
Ad=\eta\sin\theta |\vec{B}|  \text{ and }
A^2+d^2=\eta^2+|\vec{B}|^2+1.
\end{align*}
We solve
\begin{align*}
A^2=\frac{\eta^2\sin^2\theta }{ \eta^2 \sin^2\theta-d^2}(\eta^2-d^2+1)\geq \eta^2+\frac{\eta^2}{\eta^2-d^2}
\end{align*}
and the equality holds as $\sin\theta=\pm 1$. In other words, $\vec{B}$ and $\vec{C}$ are orthogonal. Under the constraints $|\vec{C}|,\ |\vec{D}|$ and $\vec{C}\cdot\vec{D}$ fixed, the minimum of $\vec{D}\cdot\vec{j}+\vec{C}\cdot\vec{c}$ occurs only if $\vec{C},\vec{D}\in$ Span$\{\vec{j},\vec{c}\}$. Denote by $\theta_0$ the angle from $\vec{j}$ to $\vec{c}$ and by $\vartheta$ the angle from $\vec{j}$ to $\vec{D}$. Then
\begin{align*}
\vec{C}\cdot\vec{c}+\vec{D}\cdot\vec{j}&=d|\vec{j}|\cos\vartheta+\eta|\vec{c}|\cos (\vartheta\pm \frac{\pi}{2}-\theta_0)\\
                                       &=\left(d|\vec{j}|\pm \eta|\vec{c}|\sin\theta_0 \right) \cos\vartheta \mp \left(\eta|\vec{c}|\cos\theta_0 \right)\sin\vartheta\\
                                       &\geq -\sqrt{d^2|\vec{j}|^2+\eta^2|\vec{c}|^2\pm 2\eta d|\vec{j}||\vec{c}|\sin\theta_0}\\
                                       &\geq -\sqrt{d^2|\vec{j}|^2+\eta^2|\vec{c}|^2+ 2\eta d|\vec{j}\times\vec{c}|}
\end{align*}
and the equality holds when $\vec{C},\vec{D}$ and $\vec{c},\vec{j}$ have the same orientation and 
\begin{equation}\label{jd_angle}
\begin{split}
\cos\vartheta &=-\frac{d|\vec{j}|+\eta |\sin\theta_0\vec{c}|}{\sqrt{d^2|\vec{j}|^2+\eta^2|\vec{c}|^2+ 2\eta d|\vec{j}\times\vec{c}|}}\\
\sin\vartheta &=\textup{sgn}(\sin\theta_0) \frac{\eta |\vec{c}|\cos\theta_0}{\sqrt{d^2|\vec{j}|^2+\eta^2|\vec{c}|^2+ 2\eta d|\vec{j}\times\vec{c}|}}.
\end{split}
\end{equation}

Let $d=\eta\cos\phi$. Then
\begin{align*}
Ae+ \vec{C} \cdot \vec{c} + \vec{D} \cdot \vec{j} &\geq \left( \sqrt{\eta^2+ \sin^{-2}\phi}\right) e -\eta\sqrt{\cos^2\phi |\vec{j}|^2+|\vec{c}|^2+2|\vec{c}\times\vec{j}| \cos\phi} \\
&\geq \sin^{-1}\phi \sqrt{e^2-( |\vec{j}|^2 \cos^2\phi+|\vec{c}|^2+2|\vec{c}\times\vec{j}| \cos\phi ) }\\
&=\sqrt{|\vec{j}|^2+\frac{e^2-|\vec{j}|^2-|\vec{c}|^2}{\sin^2\phi}-\frac{2|\vec{c}\times\vec{j}|\cos\phi}{\sin^2\phi}}
\end{align*}
where the second inequality is an equality when
\begin{equation}\label{length_b}
\eta=\sin^{-1}\phi \frac{e}{\sqrt{e^2-(|\vec{j}|^2 \cos^2\phi  +|\vec{c}|^2+2 |\vec{c}\times\vec{j}| \cos\phi) }}.
\end{equation}

Denote $P=e^2-|\vec{j}|^2-|\vec{c}|^2,\ R=2|\vec{c}\times\vec{j}|$. By our assumption $P>R$ and 
\[\frac{P}{\sin^2\phi}-\frac{R\cos\phi }{\sin^2\phi}\] attains its minimum at $\cos\phi=\frac{P-\sqrt{P^2-R^2}}{R}$ with value $\frac{1}{2}\left(P+\sqrt{P^2-R^2}\right)$. Thus 
\begin{align*}
\left( Ae+ \vec{C} \cdot \vec{c} + \vec{D} \cdot \vec{j} \right)^2\geq \frac{1}{2}\left(e^2+|\vec{j}|^2-|\vec{c}|^2+\sqrt{(e^2-|\vec{j}|^2-|\vec{c}|^2)^2-4|\vec{c}\times\vec{j}|^2}  \right).
\end{align*}
The equality holds when $\cos\phi=\frac{P-\sqrt{P^2-R^2}}{R}$, $|\vec{C}|$ is given by (\ref{length_b}), $\vartheta$ is given by (\ref{jd_angle}), $\vec{C},\vec{D}$ and $\vec{c},\vec{j}$ span the same plane with the same orientation, and $\vec{B},\vec{C}$ are orthogonal. Suppose $e = \sqrt{|\vec{j}|^2+|\vec{c}|^2+2|\vec{j}\times\vec{c}|}$. Then by replacing $e$ by $e+\epsilon$ and letting $\epsilon\to 0$ the result still follows.\\

\textbf{Case 2:} General $\vec{p}.$\\
From Theorem \ref{rigidity_Chrusciel}, $e\geq |\vec{p}|$. Suppose $e>|\vec{p}|$ then by a coordinate change we can make $\vec{p}=0$. The result follows since $\alpha$ and $\beta$ are invariant under the $SO(3,2)$ action. For $e=|\vec{p}|$, we can also replace $e$ by $e+\epsilon$ and let $\epsilon\to 0$.
\end{proof}

We now prove Theorem \ref{pmt}.
\bigskip

\noindent {\bf Theorem \ref{pmt} }   {\it Let $(M, g, k)$ be an asymptotically AdS initial data set that satisfies the dominant energy condition as in \cite{Chrusciel-Maerten-Tod}. The rest mass of $(M, g, k)$ is non-negative. It vanishes if and only if $(M, g, k)$  is the data of a hypersurface in the AdS space-time $AdS^{3,1}$.
}

\bigskip
\begin{proof}
We may assume  $e>|\vec{p}|$ since  $e=|\vec{p}|$ is precisely the second condition in  Theorem \ref{rigidity_Chrusciel}  such that the data comes from a hypersurface in the AdS spacetime.
By boosting, we can further assume $\vec{p}=0$. In this case, the rest mass is simply
\[ m^2=\frac{1}{2}\left(e^2+|\vec{j}|^2-|\vec{c}|^2+\sqrt{(e^2-|\vec{j}|^2-|\vec{c}|^2)^2-4|\vec{c}\times\vec{j}|^2}  \right) .\]
As a result, $m=0$ implies $\vec{j}=0$ since  $e^2\geq |\vec{j}|^2+|\vec{c}|^2+2|\vec{j}\times\vec{c}|$ by Theorem \ref{rigidity_Chrusciel}. We conclude that 
$\vec{p} = \vec{j} = 0$ and $e = |\vec{c}|$. It follows that the data comes from an AdS hypersurface since this is a special case of the third condition in the rigidity statement of  Theorem \ref{rigidity_Chrusciel}.

\end{proof}

\section{Algebraic interpretations and comparisons with previous results}

In this section, we explain and compare the results as algebraic properties of the Lie algebra $\underline{so}(3,2)$ and the spinor representation
of the Clifford algebra $Cl(3,1)$, which is isomorphic to $\mathbb{C}^4$. 
Let $\gamma_0, \gamma_1, \gamma_2, \gamma_3$ be an orthonormal basis in $\R^{3,1}$ and identify them with elements in the Clifford Algebra
$Cl(3,1)$ which acts on $\mathbb{C}^4$ by the Clifford multiplication.

Witten's formula gives the following algebraic expression (compare with \cite[(4.6)]{Gibbons-Hull-Warner} and \cite[page 11]{Chrusciel-Maerten-Tod}): 
\begin{equation}\label{energy}\left((e\gamma_0-p_j \gamma_j+c_k(i \gamma_{0}\gamma_k)+j_m \frac{i}{2}\epsilon^{mkl}\gamma_{k}\gamma_l)\psi,\psi\right)\end{equation}
where $(e, \vec{p}, \vec{c}, \vec{j})$ are physical quantities. $(e\gamma_0-p_j \gamma_j+c_k(i \gamma_{0}\gamma_k)+j_m \frac{i}{2}\epsilon^{mkl}\gamma_{k}\gamma_l)\in Cl(3,1)$ acts on $\psi\in \mathbb{C}^4$.  The positive energy theorem shows that \eqref{energy} is non-negative for any $\psi \in \mathbb{C}^4$. Note that since $\gamma_0^2=1$ and $(\gamma_0 \psi, \phi)=\langle \psi, \phi\rangle$, we can rewrite \eqref{energy} in terms of the standard Hermitian form $\langle \cdot, \cdot\rangle$ on $\mathbb{C}^4$:
\begin{equation}\langle (e-p_j (\gamma_0 \gamma_j)+c_k(i \gamma_k)+j_m (\frac{i}{2}\epsilon^{mkl}\gamma_0 \gamma_{k}\gamma_l)\psi,\psi\rangle.\end{equation}

Gibbons et al. used the spinor representation \cite[(4.7)]{Gibbons-Hull-Warner} to identify $ e\gamma_0-p_j \gamma_j+c_k(i \gamma_{0}\gamma_k)+j_m \frac{i}{2}\epsilon^{mkl}\gamma_{k}\gamma_l $ with an element $J_{AB}$ in $\underline{so}(3, 2)$, and proceeded to further identify this with a real $4\times 4$ matrix $J_{ab}$ \cite[(A, 20)]{Gibbons-Hull-Warner} through the isomorphism $\underline{so}(3, 2)\cong \underline{sp}(4)$. The positive energy theorem is equivalent to $J_{ab}$ being non-negative definite. In \cite{Gibbons-Hull-Warner},  Gibbons et al derived the following inequalities under the stronger assumption that $J_{ab}$ is positive definite (or no Killing spinor).
\begin{align*}
&C_2 >0\\
&\frac{1}{2}C_2^2> C_4>\frac{1}{4}C_2^2
\end{align*}
where $C_2=J^A_{\ \ B} J^B_{\ \ A}$ and $C_4=J^A_{\ \ B} J^B_{\ \ C} J^C_{\ \ D} J^D_{\ \ A}$. In terms of $K^*=(e,\vec{p},\vec{c},\vec{j})$, $C_2=2\alpha (K^*)$ and $C_4=\alpha (K^*)^2 + \beta(K^*)$. The above inequalities can be rewritten as
\begin{align*}
& \alpha(K^*) >0\\
& \alpha(K^*)^2 > \beta(K^*) >0.
\end{align*}
Also, they can choose a suitable frame such that $\vec{p}$ and $\vec{c}$ vanish. In this case they got the expression of $e$ which is the same as the rest mass $m$ in (\ref{rest_mass}). One can also get $e=m$ from (\ref{alpha_beta}) assuming $\vec{p}=\vec{c}=0$.

Chru\'sciel et al. studied the more general case when $J_{ab}$ ($Q$ defined on \cite[page 11]{Chrusciel-Maerten-Tod}) is non-negative definite and derived several inequalities among $(e, \vec{p}, \vec{c}, \vec{j})$. In particular, they discussed the case when $J_{ab}$ or $Q$ has non-trivial zero eigenvalues and how such conditions imply the rigidity property that the physical data arises from the AdS spacetime. In the case $\vec{p}=0$ they obtained an inequality among energy, center of mass and angular momentum:
\begin{equation}
e\geq \sqrt{|\vec{c}|^2+|\vec{j}|^2+2|\vec{c}\times\vec{j}|} \tag{\ref{rigidity_Chrusciel}}
\end{equation}
which implies $\alpha (K^*)\geq 0$ and $\alpha(K^*)^2\geq \beta (K^*)\geq 0$ when $\vec{p}=0$.

Here we take a different approach from \cite{Gibbons-Hull-Warner} and \cite{Chrusciel-Maerten-Tod} and interpret \eqref{energy} as a pairing 
\[eA+\vec{p}\cdot \vec{B}+\vec{c}\cdot \vec{C}+\vec{j}\cdot \vec{D}\] of the conserved quantities $(e, \vec{p}, \vec{c}, \vec{j})$ and 
\begin{align*}
A=(\gamma_0\psi,\psi),&\ \   B_j=- (\gamma_j\psi,\psi)\\
C_j=(i\gamma_{0j}\psi,\psi) &,\ \ D_j= (\frac{i}{2}\epsilon^{jkl}\gamma_{kl}\psi,\psi).
\end{align*}
In our approach, $(A, \vec{B}, \vec{C}, \vec{D})$ is viewed an element in $\underline{so}(3, 2)$ through the map $\mathbb{C}^4\rightarrow \underline{so}(3, 2)$ defined in Corollary \ref{spinor_map}. The expression is considered as a paring of $(e, \vec{p}, \vec{c}, \vec{j}) \in \underline{so}(3, 2)^*$ with $(A, \vec{B}, \vec{C}, \vec{D})$ for an $(A, \vec{B}, \vec{C}, \vec{D})$ in the image of the map, which is identified with spinor Killing fields $\mathcal{S}$. We utilize the positivity of these expressions to study observer energy, or the pairing with elements in the set of observer Killing fields $\mathcal{O}$.

\end{document}